\documentclass[a4paper,12pt]{amsart}

\newtheorem{thm}{Theorem}[section]
\newtheorem{pro}[thm]{Proposition}
\newtheorem{cor}[thm]{Corollary}

\theoremstyle{definition}
\newtheorem{rem}[thm]{Remark}

\newtheorem{exam}[thm]{Example}
\usepackage{amssymb}
\usepackage{amsmath}

\newcommand{\R}{\mathbb R}
\def\ce{Ces\`{a}ro }

\def\ces{\mathcal{C}}
\def\linf{L^\infty}
\def\cx{[\mathcal{C},X]}
\def\lmx{L^1(m_X)}
\def\lmxv{L^1(|m_X|)}
\def\wlmx{L_w^1(m_X)}
\def\ele{L^1}
\def\elp{L^p}

\def\elexp{L_{\textrm{exp}}}

\def\vfi{\varphi}
\def\fix{\varphi_{X}}
\def\imx{I_{m_X}}
\def\mx{m_X}
\def\mxv{|m_X|}
\def\marf{{M(\varphi)}}
\def\laf{{\Lambda(\varphi)}}
\def\lnu{{L^1(m)}}
\def\lnuv{{L^1(|m|)}}

\begin{document}

\title[Abstract  Ces\`{a}ro spaces]
{Abstract  Ces\`{a}ro spaces: Integral representations}

\author{Guillermo P. Curbera}
\address{Facultad de Matem\'aticas \& IMUS,  Universidad de Sevilla,
Aptdo.  1160,  Sevilla 41080, Spain}
\email{curbera@us.es}

\author{Werner J. Ricker}
\address{Math.--Geogr. Fakult\"at, Katholische Universit\"at
Eichst\"att--Ingolstadt, D--85072 Eichst\"att, Germany}
\email{werner.ricker@ku.de}

\thanks{The first author acknowledges the support of the
\lq\lq International Visiting Professor Program 2015\rq\rq, via the
Ministry of Education, Science and Art, Bavaria (Germany).}

\date{\today}



\begin{abstract}
The \ce function spaces $Ces_p=[\ces,L^p]$, $1\le p\le\infty$, have received renewed attention in recent years.
Many properties of $[\ces,L^p]$ are known. Less is known about $\cx$ when the \ce operator
takes its values in a rearrangement invariant (r.i.) space $X$ other than $L^p$. In this paper we
study the spaces $\cx$ via the methods of vector measures and vector integration. These techniques allow us to
identify the absolutely continuous part of $\cx$ and the Fatou completion of $\cx$;
to show that $\cx$ is never reflexive and never r.i.; to identify when $\cx$ is  weakly sequentially complete,
when it is isomorphic to an AL-space, and when it has
the Dunford-Pettis property. The same techniques are used to analyze the operator $\ces\colon\cx\to X$;
it is never compact but,  it can be completely continuous.
\end{abstract}

\maketitle


\section*{Introduction}


\ce function spaces  have attracted  much attention in recent times;
see for example the papers  \cite{astashkin-maligranda0}, \cite{astashkin-maligranda1},
\cite{astashkin-maligranda2} by Astashkin and Maligranda and \cite{lesnik-maligranda-1},
\cite{lesnik-maligranda-2} by L\'{e}snik and Maligranda
and the references therein. These spaces arise when studying the behavior,
in certain function spaces, of the \ce operator
\begin{equation*}
\ces:f\mapsto \ces(f)(x):=\frac{1}{x} \int_0^x f(t)\,dt.
\end{equation*}
A classical result of Hardy motivated the study of the
operator $\ces$ in the $\elp$ spaces,
thereby leading to the spaces  $Ces_p:=\{f: \ces(|f|)\in \elp\}$. It was then natural
to extend the investigation to the so called
abstract \ce spaces $\cx$, where the role of  $\elp$
is replaced by a more general function space $X$, namely, the Banach
function space (B.f.s.)
\begin{equation*}
\cx:=\big\{f: \ces(|f|)\in X\big\},
\end{equation*}
equipped with the norm
\begin{equation*}
\|f\|_{\cx}:=\|\ces(|f|)\|_X,\quad f\in\cx.
\end{equation*}
We will focus our attention on those spaces $X$ which
are rearrangement invariant (r.i.)  on $[0,1]$.

It is known that $[\ces,L^p]=Ces_p$ is not reflexive, \cite[Theorem 1, Remark 1]{astashkin-maligranda0}.
In  Theorem \ref{reflexive} it is shown that $\cx$ is never reflexive.
This  result is established via techniques from a different area. It turns out, for
every r.i.\ space $X\not=\linf$, that the $X$-valued set function
\begin{equation*}
\mx:A\mapsto \mx(A):=\mathcal{C}(\chi_A),
\quad A\subseteq [0,1] \text{ measurable},
\end{equation*}
is $\sigma$-additive, i.e., it is a \textit{vector measure}. This fact can be
successfully used for studying the function space $\cx$. Indeed,
the norm of  $\cx$ is not necessarily
absolutely continuous (a.c.). Actually, the a.c.\ part
$\cx_a$ of $\cx$ is precisely the well understood space $\lmx$ consisting of all the $\mx$-integrable functions
(in the sense of Bartle, Dunford and Schwartz, \cite{bartle-dunford-schwartz}). Moreover,
$\cx$ need not  have the Fatou property. It turns
out that the Fatou completion $\cx''$ of $\cx$
is precisely the space $\wlmx$ consisting of all the weakly $\mx$-integrable functions.

A further relevant point is that the integration operator $\imx\colon\lmx\to X$ given by
$f\mapsto\int f\,d\mx$ is precisely the restriction to $\cx_a$ of the \ce operator
$\ces\colon\cx\to X$.
Moreover, $\lmx$ is the \textit{largest} B.f.s.\ over $[0,1]$ with a.c.\  norm
on which $\ces$ acts with values in $X$.
In addition, the scalar variation measure
$|\mx|$ of the vector measure $\mx$ is always $\sigma$-finite and  possesses a strongly measurable,
Pettis integrable density $F\colon[0,1]\to X$ relative to Lebesgue measure.
A relevant feature for the operator $\ces$ (which a priori is only given by a 
pointwise expression on $\cx$) is that  \textit{integral representations} become
available. First, for  $\ces$ restricted to $\cx_a$, such a representation is given by
\begin{equation}\label{representationBDS}
\ces(f)=\int_{[0,1]} f\,d\mx,\quad f\in\lmx=\cx_a,
\end{equation}
via the Bartle-Dunford-Schwartz integral for vector measures. Actually, 
it turns out specifically for $\mx$ that
\begin{equation}\label{representation}
\ces(f)=\int_{[0,1]} f(y)\,F(y)\,dy,\quad f\in\lmx ,
\end{equation}
which is defined more traditionally   as  a Pettis  integral. Furthermore, for the  class of r.i.\ spaces $X$ where the variation
measure  $\mxv$ is  finite, the representation \eqref{representation} 
when restricted to $\lmxv$ is actually  given via a
\textit{Bochner integrable density} $F$.

The paper is organized as follows.

In Section 1 we present the preliminaries on Banach function spaces, rearrangement invariant spaces and vector integration
that are needed in the sequel.

Section 2 is devoted to establishing the main properties of the vector measure $\mx$.
A large class of r.i.\ spaces $X$ for which $\mxv$ is a finite measure is identified;
see Proposition \ref{variation-L} and Corollary \ref{condition var}.

In Section 3 the study of the space $\cx$ is undertaken with the vector measure $\mx$ and its space of
integrable function $\lmx$ as main tools. As mentioned above, in
Theorem \ref{reflexive} it is proved that $\cx$ is never reflexive. It is also
established as part of that result
that  $\cx$ fails to be r.i.\ (this was proved for $[\ces,L^p]$ in
\cite[Theorem 1]{astashkin-maligranda1} and conjectured in \cite[Remark 3]{lesnik-maligranda-1}).
The problem of when $\cx$ is order isomorphic to an AL-space,
that is, to a Banach lattice where the norm is additive over disjoint functions, is also considered.
It is shown (cf. Theorem \ref{L1}(a)), for a large class of Lorentz spaces $\laf$, that
$[\ces,\laf]$ is order isomorphic to $L^1(|m_{\laf}|)$ with $|m_{\laf}|$ a finite, non-atomic measure.
Crucial for the proof of the existence of this order isomorphism
is an identification, due to L\'esnik and Maligranda, \cite{lesnik-maligranda-1}, of the associate
space $\cx'$ of the B.f.s.\ $\cx$ (under some restrictions on the r.i.\ space $X$).

In Section 4 we analyze the operator $\ces\colon\cx\to X$. The identification of the
restriction of $\ces$, via $\imx$, is used to show that the operator
$\ces\colon\cx\to X$ is never compact; see Proposition \ref{compact}.
For r.i.\ spaces $X$ satisfying $X\subseteq\lmxv$,
which forces both $\mx$ to have finite variation and  $\ces\colon X\to X$ to act boundedly, it follows
(cf.\ Proposition \ref{cc lmxv}) that $\ces\colon X\to X$ is necessarily completely continuous.
This result is quite useful in view of the fact that $\ces\colon X\to X$ is never compact (whenever
it is a bounded operator). The complete continuity of the restricted integration operator
$\imx\colon\lmxv\to X$ can be `lifted' to the complete continuity of  $\ces\colon\cx\to X$,
under some conditions on the r.i.\ space $X$; see Proposition \ref{cc lmxv}. This property
of $\ces\colon\cx\to X$ is related to $\cx$ being order isomorphic to an AL-space;  see Proposition \ref{cc lmx}.
The section ends with another extension of a result valid for $X=L^p$. It was
shown in \cite[\S 6, Corollary 1]{astashkin-maligranda1} that the spaces $Ces_p$, $1<p<\infty$, fail to
have the Dunford-Pettis property. This result is extended to include
all reflexive r.i.\ spaces $X$ having a non-trivial upper Boyd index; see Proposition \ref{4.7}.

In the final section we discuss in fine detail the role of the Fatou property in relation to $\cx$, and
derive some consequences for $\cx$; see Proposition \ref{5.2}.

We only consider r.i.\ spaces $X\not=\linf$ because $Ces_\infty=[\ces,\linf]$, known as the
Korenblyum-Kre\u{\i}n-Levin space, has already been thoroughly investigated; see
\cite{astashkin-maligranda1}, \cite{astashkin-maligranda2} and the references therein.


\section{Preliminaries}


A \textit{Banach function space} (B.f.s.) $X$ on [0,1] is a
Banach space  of classes of measurable functions on [0,1] satisfying
the ideal property, that is, $g\in X$ and $\|g\|_X\le\|f\|_X$
whenever $f\in X$ and $|g|\le|f|$ $\lambda$--a.e.,  where $\lambda$ is
the Lebesgue measure on [0,1].
The \textit{associate space} $X'$  of $X$ consists  of all
functions $g$ satisfying $\int_0^1|f(t)g(t)|\,dt<\infty$, for every
$f\in X$. The space $X'$ is a subspace of the Banach space dual $X^*$ of $X$.
The \textit{absolutely continuous} (a.c.)  part $X_a$ of $X$ is  the space of all functions
$f\in X$ satisfying $\lim_{\lambda(A)\to0}\|f\chi_A\|_X=0$; here $\chi_A$
is the characteristic function of the set $A\in\mathcal{M}$, with $\mathcal{M}$ denoting the $\sigma$-algebra
of all Lebesgue measurable subsets of $[0,1]$.
If $\linf\subseteq X_a$,  then the closure of $\linf$ in $X$ coincides with $X_a$ and
$(X_a)'=X'$. The space $X$ is said to have a.c.\ norm if $X=X_a$. In this case, $X'=X^*$.
The space $X$ satisfies the \textit{Fatou property} if $\{f_n\}\subseteq X$ with
 $0\le f_n\le f_{n+1}\uparrow f$ $\lambda$-a.e.\ and $\sup_n\|f_n\|_X<\infty$
imply that $f\in X$ and $\|f_n\|_X\to\|f\|_X$.  The second  associate space $X''$ of $X$ is defined as $X''=(X')'$.
The space $X$ has the Fatou property if and only if $X''=X$. Unless specifically stated, it is not
assumed that the Fatou property holds in $X$.

A \textit{rearrangement invariant} (r.i.) space $X$ on [0,1] is a
B.f.s.\  on $[0,1]$ such that if $g^*\le f^*$ and $f\in X$,  then $g\in X$ and $\|g\|_X\le\|f\|_X$.
Here $f^*$ is the \textit{decreasing rearrangement} of $f$, that is, the
right continuous inverse of its distribution function:
$\lambda_f(\tau):=\lambda(\{t\in [0,1]:\,|f(t)|>\tau\})$.
The associate space $X'$ of a r.i.\ space $X$ is again a r.i.\ space.
A r.i.\ space $X$ satisfies $\linf\subseteq X\subseteq \ele$.
If $X\not=\linf$,  then $(X_a)'=X'$.
The \textit{fundamental function} $\fix$ of $X$ is defined via
$\varphi_X(t):=\|\chi_{[0,t]}\|_X$. For $X\not=\linf$ we have $\lim_{t\to0}\fix(t)=0$,
\cite[Lemma 3, p.220]{rodin-semenov}.

Important classes of r.i.\ spaces are the Lorentz and Marcinkiewicz
spaces. Let $\varphi\colon[0,1]\to[0,\infty)$  be an increasing, concave
function with $\varphi(0)=0$. The Lorentz space $\Lambda(\varphi)$
consists of all measurable functions $f$ on [0,1] satisfying
\begin{equation*}
\|f\|_{\Lambda(\varphi)}:=\int_0^1f^*(s)\,d\varphi(s) <\,\infty.
\end{equation*}
Let $\varphi\colon[0,1]\to[0,\infty)$  be a
quasi-concave function, that is, $\varphi$ is increasing, the function $t\mapsto\varphi(t)/t$
is decreasing and $\varphi(0)=0$. The Marcinkiewicz space $\marf$  consists
 of all measurable functions  $f$ on [0,1] satisfying
\begin{equation*}
\|f\|_{\marf}:=\sup_{0<t\le 1}\, \frac{\varphi(t)}{t}\,\int_0^tf^*(s)
\, ds<\infty.
\end{equation*}

The Marcinkiewicz space $M(\varphi)$ and the Lorentz space $\Lambda(\varphi)$
are, respectively, the largest and the smallest r.i.\ spaces having the
fundamental function $\varphi$. That is, for any r.i.\ space $X$  we have
$\Lambda(\fix)\subseteq X\subseteq M(\fix)$.
The associate space $\laf'=M(\psi)$ and $\marf'=\Lambda(\psi)$, for
$\psi(t):=t/\varphi(t)$. In the notation of \cite[p.144]{krein-petunin-semenov},
observe that $\marf=M_\psi$.

If $\phi$ is a positive function defined on [0,1], then  its lower and upper
dilation indices are, respectively, defined by
\begin{equation*}
\gamma_\phi := \lim_{t\to 0^+} \frac{\log\big(\sup_{\,0<s\le 1}
\frac{\phi(st)}{\phi(s)}\big)}{\log t}, \qquad
\delta_\phi := \lim_{t\to +\infty} \frac{\log\big(\sup_{\,0<s\le
1/t} \frac{\phi(st)}{\phi(s)}\big)}{\log t}.
\end{equation*}
For a quasi-concave function $\vfi$ it is known that $0\le \gamma_\vfi \le \delta_\vfi\le1$.
Whenever $\delta_\varphi<1$ the following equivalence for the above norm  in
$\marf$ holds (see \cite[Theorem II.5.3]{krein-petunin-semenov}):
\begin{equation}\label{norm-marz}
\|f\|_{\marf}\asymp\sup_{0<t\le 1}\, \varphi(t)f^*(t).
\end{equation}
The notation   $A\asymp B$ means that there exist constants $C>0$ and
$c>0$ such that $c{\cdot}A\le B\le C{\cdot}A$.
For further details concerning
r.i.\ spaces we refer to \cite{bennett-sharpley}, \cite{krein-petunin-semenov},
\cite{lindenstrauss-tzafriri}; care should be taken with  \cite{bennett-sharpley}
as all r.i.\ spaces there are assumed to have the Fatou property.
General references for B.f.s.' include \cite{okada-ricker-sanchez}, \cite[Ch.15]{zaanen1}.

\bigskip

We recall briefly  the theory of integration of real functions with
respect to a vector measure, initially due to  Bartle, Dunford and
Schwartz, \cite{bartle-dunford-schwartz}. Let $(\Omega,\Sigma)$ be a measurable
space, $X$ a Banach space  and $m\colon\Sigma\to X$ a $\sigma$-additive vector measure.
For each $x^*\in X^*$, denote the $\R$--valued  measure $A\mapsto
\langle x^*,m(A)\rangle$ by $x^*m$ and its variation  measure
by $|x^*m|$. A measurable function $f\colon\Omega\to\R$ is said to be
\textit{integrable with respect to} $m$ if $f\in L^1(|x^*m|)$, for
every $x^*\in X^*$, and for each $A\in\Sigma$ there exists a
vector in $X$ (denoted by $\int_Af\,dm$) satisfying
$\langle\int_Af\,dm,x^*\rangle=\int_Af\,d\, x^*m$, for every
$x^*\in X^*$.   The $m$--integrable functions
form a linear space in which
\begin{equation}\label{norm-m}
 \|f\|_{L^1(m)} : =\sup\left\{\int |f|\,
d|x^*m| \colon x^*\in X^*, \|x^*\|\le1\right\}
\end{equation}
is a seminorm. A  set $A\in\Sigma$ is called $m$--\textit{null} if $|x^*m|(A)=0$ for every
$x^*\in X^*$. Identifying functions which differ only in a $m$--null
set, we obtain a Banach space (of classes) of $m$--integrable
functions, denoted by $L^1(m)$. It is a B.f.s.\
for the $m$--a.e.\ order and has a.c.\  norm.
The simple functions are dense in
$L^1(m)$ and the space $L^\infty(m)$ of all $m$--essentially bounded functions is
contained in $L^1(m)$. The \textit{integration operator} $I_{m}$ from
$L^1(m)$ to $X$ is defined by $f\mapsto\int f\,dm:=\int_\Omega f\,dm$. It is
continuous, linear and has operator norm at most one. No assumptions have
been made on the \textit{variation measure} $|m|$ of  $m$ (cf.  \cite[\S3.1]{okada-ricker-sanchez}) in the
definition of $L^1(m)$. In general $\lnuv\subseteq\lnu$.
We will repeatedly use the following property:
let $Y$ be the closed linear subspace of $X$ generated  by the
range $m(\Sigma)$ of the vector measure $m$. Then
$\lmx=L^1(m_Y)$ and $\lmxv=L^1(|m_Y|)$, where $m_Y\colon\Sigma\to Y$
is given by $m_Y(A):=m_X(A)$ for all $A\in\Sigma$.

The B.f.s.'  $\lnu$ can be quite different to the
classical $L^1$--spaces of scalar measures and may be difficult to
identify explicitly.  Indeed, every Banach
lattice with a.c.\  norm and having a weak unit (e.g. $L^2([0,1])$)
is the $L^1$--space of some vector measure, \cite[Theorem
8]{curbera1}.
For further details concerning $L^1(m)$ and $I_m$ see, for example, \cite[Ch.3]{okada-ricker-sanchez}
and the references therein.


\section{The vector measure induced by $\ces$}


The vector measure associated to the \ce operator is defined by
\begin{equation*}\label{measure-m}
m \colon A\longmapsto m(A):=\mathcal{C}(\chi_A),\qquad A \in\mathcal{M}.
\end{equation*}
Since $\mathcal{C}$ maps $\linf$ into itself, we have $m(\mathcal{M})\subseteq \linf$.
So, $m$ is a well defined,  finitely additive vector
measure with values in $\linf$ but, it is not $\sigma$-additive
as an $\linf$-valued measure, \cite{ricker}.
For every r.i.\ space $X$ we have $\linf\subseteq X$. Accordingly, $m$  is also
well defined and finitely additive  with values in $X$. We will denote $m$
by $m_X$ whenever it is necessary to indicate that the values of $m$ are considered to be in $X$.


\begin{thm}\label{measure}
Let $X\not=\linf$ be a r.i.\ space.
\begin{itemize}
\item[(a)] The measure $m_X$ is $\sigma$-additive.
\item[(b)] The measure $m_X$ has  a strongly measurable,
$X$-valued, Pettis $\lambda$-integrable density
$F$ given by
\begin{equation}\label{density}
F: y\in[0,1]\mapsto F_y\in X \textrm{ with }  F_y(x)=:\frac1x \chi_{[y,1]}(x),\quad 0<x\le1.
\end{equation}
\item[(c)] The measure $m_X$ has $\sigma$-finite variation given by
\begin{equation}\label{var}
\mxv(A)=\int_A\|F_y\|_X\,dy,\quad A\in\mathcal{M}.
\end{equation}

In the event that $\mx$ has finite variation, $F$ is actually Bochner $\lambda$-integrable.
\item[(d)] The range $\mx(\mathcal{M})$ of $m_X$ is a relatively compact set in $X$.
\end{itemize}
\end{thm}

\begin{proof}
(a) Let $(A_n)$ be a sequence of sets with $A_n\downarrow\emptyset$.
Then the functions $(\chi_{A_n})$ decrease pointwise to zero.
Since $\mathcal{C}$ is a positive operator, the sequence
$(\mathcal{C}(\chi_{A_n}))$ is also decreasing; by the
Dominated Convergence Theorem applied to $\chi_{A_n}\downarrow0$
it follows that $(\mathcal{C}(\chi_{A_n}))$
actually decreases to zero a.e.
Recall that $\mx(\mathcal{M})\subseteq \linf\subseteq X_a$.
But, $X_a$ has a.c.\  norm and so $\|\mathcal{C}(\chi_{A_n})\|_{X_a}\to0$.
Since the norms of $X_a$ and  $X$ coincide, we have
$\|\mathcal{C}(\chi_{A_n})\|_{X}\to0$, i.e.,  $m_X(A_n)\to0$ in $X$.

(b) Consider the $X$-valued vector function $F$ given by \eqref{density}.
It is a.e.\ well defined  since, for each $0<y\le 1$, we have
$F_y\in\linf\subset X$. To prove that it is  strongly measurable
it suffices to verify that $y\in(0,1]\mapsto F_y\in X$ is continuous. Fix $0<t<s\le 1$, in which case
\begin{equation*}
\|F_t-F_s\|_X= \left\|\frac1x\chi_{[t,s)}\right\|_X \le \frac1t\varphi_X(s-t).
\end{equation*}
Since $X\not=\linf$, it follows that $\varphi_X(s-t)\to0$ as $(s-t)\to0$.

Next we check the Pettis $\lambda$-integrability of $F$. Note that
 $F_y\in\  X_a$ for $y\in(0,1]$.
For any $0\le g\in (X_a)'$ we have, via Fubini's theorem, that
\begin{equation*}
\int_0^1 \big\langle F_y,g\big\rangle\,dy =
\int_0^1 \int_0^1 \frac1x\chi_{[y,1]}(x)g(x)\,dx\,dy
=
\int_0^1 g(x)\,dx,
\end{equation*}
which is surely finite as $(X_a)'\subseteq L^1$. Since  elements of $X^*$ restricted to $X_a$ belong to
$(X_a)^*$ and $(X_a)^*=(X_a)'$, it follows that
 $y\mapsto \langle F_y,x^*\rangle\in\ele$ for every $x^*\in X^*$.

It remains to check that $F$ is the Pettis $\lambda$-integrable density for $\mx$.
Fix $A\in\mathcal{M}$ and recall that $m_X(A)=\mathcal{C}(\chi_A)\in X_a$.
For  $0\le g\in (X_a)'$ an application of Fubini's theorem yields
\begin{eqnarray*}
\big\langle m_X(A),g\big\rangle&=&\int_0^1g(x)
\left(\int_0^1\frac1x\chi_{[0,x]}(y)\chi_A(y)\,dy\right)\,dx
\\
&=&
\int_A \int_0^1g(x)\frac1x\chi_{[y,1]}(x)\,dx\,dy
\\
&=&\int_A \Big\langle \frac1x\chi_{[y,1]}(x), g(x)\Big\rangle\,dy
\\
&=&
 \int_A \Big\langle F_y, g\Big\rangle\,dy .
\end{eqnarray*}
Since this is valid for every $0\le g\in(X_a)'$ and $(X_a)^*=(X_a)'$, it follows
that $F\colon y\mapsto F_y$ is Pettis $\lambda$-integrable with
\begin{equation}\label{aaa}
\int_A F_y\,dy:=m_X(A)\in X_a\subset X,\quad A\in\mathcal{M} .
\end{equation}

(c) Fix $0<a<1$. Consider the  measure $m_X$ restricted to  $[a,1]$.   Since $F$
is continuous on the compact set $[a,1]$, we have
$\int_{[a,1]} \|F_y\|_X\,dy<\infty$. According to \eqref{aaa}, $y\mapsto F_y$ is
then a Bochner $\lambda$-integrable density
for $\mx$ on $[a,1]$. Accordingly,
\begin{equation}\label{equality}
\mxv(A)
=\int_A \|F_y\|_X\,dy,\quad A\in\mathcal{M},\quad A\subseteq[a,1] .
\end{equation}

Let now $A\in\mathcal{M}$. Set $A_n:=[1/n,1]\cap A$, for $n\ge2$,
in which case $\mxv(A_n)<\infty$.
Observing that $\chi_{A_n}(y)\|F_y\|_X\uparrow\chi_A(y)\|F_y\|_X$
$\lambda$-a.e.\ it follows from \eqref{equality}
and the $\sigma$-additivity of $\mxv$ that
$$
\mx(A)=\lim_n\mx(A_n)=\lim_n\int_{A_n}\|F_y\|_X\,dy = \int_{A}\|F_y\|_X\,dy .
$$
This establishes \eqref{var} and the $\sigma$-finiteness of the variation.

In the event that $\mx$ has finite variation,  \eqref{var} implies that $y\mapsto\|F_y\|_X$ belongs
to $L^1$ and hence $F$, being strongly measurable, is Bochner $\lambda$-integrable.

(d) Set $D_n=(1/2^n,1/2^{n-1}]$, for $n\ge1$. Then for each $n\ge1$
we have $|m_X|(D_n)<\infty$. Moreover,  the
density $F$ is Bochner $\lambda$-integrable over each $D_n$. Hence, the
range $m_X(\mathcal{M}_{D_n})$ is relatively compact in $X$, \cite[p.148]{okada-ricker-sanchez}, where
$\mathcal{M}_{D_n}:=\{A\in\mathcal{M}:A\subseteq D_n\}$. Thus,
$$
m_X(\mathcal{M})=\sum_{n=1}^\infty m_X(\mathcal{M}_{D_n})
:=\Big\{\sum_{n=1}^\infty f_n:f_n\in\mx(\mathcal{M}_{D_n}) \textrm{ for } n\in\mathbb{N}\Big\}.
$$
Arguing as in the proof of Corollary 2.43 (see also part II of Proposition 3.56) in \cite{okada-ricker-sanchez} we deduce
that $m_X(\mathcal{M})$ is relatively compact in $X$.
\end{proof}

\begin{rem}\label{2.2}
(a) The $\sigma$-finiteness of $\mxv$ also follows from a general result on Pettis integration,
\cite[Proposition 5.6(iv)]{van-dulst}. Since $([0,1],\mathcal{M},\lambda)$ is a perfect measure space,
the relative compactness of $\mx(\mathcal{M})$ in $X$ is also a general result (due
to C. Stegall), \cite[Proposition 5.7]{van-dulst}.

(b) It follows from \eqref{var} that $\lambda$ and $\mx$ have the same null sets.
\end{rem}


For certain r.i.\ spaces $X$ it is possible to compute $\mxv$ precisely.

\begin{pro}\label{variation-L}
For the Lorentz space $\Lambda(\varphi)$ we have
\begin{equation}\label{normL}
\|F_y\|_{\Lambda(\varphi)}= \int_0^{1-y} \frac{\varphi'(t)}{t+y}\,dt,\quad y\in(0,1],
\end{equation}
and
\begin{equation*}
|m_{\Lambda(\varphi)}|([0,1]) =\int_0^1\log(1/t)\,\varphi'(t)\,dt .
\end{equation*}

Consequently, $m_{\Lambda(\varphi)}$ has finite variation precisely
when $\log(1/t)\in\Lambda(\varphi)$ and, in that case,
$|m_{\Lambda(\varphi)}|([0,1])=\big\|\log(1/t)\big\|_{\Lambda(\varphi)}$.
\end{pro}

\begin{proof}
For $y\in(0,1]$ the   decreasing rearrangement of $F_y(\cdot)$ is given by
\begin{equation}\label{density*}
(F_y)^*(t)=F_y(t+y)=\frac{1}{t+y}\chi_{[0,1-y]}(t),\quad 0\le t\le1.
\end{equation}
It follows  that
\begin{equation*}
\|F_y\|_{\Lambda(\varphi)}= \int_0^{1-y} \frac{\varphi'(t)}{t+y}\,dt,\quad y\in(0,1].
\end{equation*}
Then, from \eqref{var} we can conclude that
\begin{equation*}
|m_{\Lambda(\varphi)}|(A)=
\int_A\|F_y\|_{\Lambda(\varphi)}\,dy =
\int_A \left(\int_0^{1-y} \frac{\varphi'(t)}{t+y}\,dt\right)\,dy .
\end{equation*}
For $A=[0,1]$ an application of  Fubini's theorem yields
\begin{equation*}
|m_{\Lambda(\varphi)}|([0,1])
=
\int_0^1 \left(\int_0^{1-y} \frac{\varphi'(t)}{t+y}\,dt\right)\,dy
=
\int_0^1\log(1/t)\,\varphi'(t)\,dt .
\end{equation*}
Since $t\mapsto\log(1/t)$ is decreasing, it is clear  that $m_{\laf}$
has finite variation precisely when $\log(1/t)\in\laf$ in which case
$|m_{\laf}|([0,1])=\|\log(1/t)\|_\laf$.
\end{proof}

\begin{exam}\label{variation-L-rem}
The Zygmund spaces of exponential integrability $L^p_{\textrm{exp}}$,
for $p>0$,  are \lq\lq close" to $\linf$; see \cite[Definition IV.6.11]{bennett-sharpley}.
The classical space  $\elexp$ (i.e. $p=1$) is a particular case. The space
$L^p_{\textrm{exp}}$ coincides with  the Marcinkiewicz space $M(\varphi_p)$ for $\varphi_p(t):=\log^{-1/p}(e/t)$.
For $X=\Lambda(\varphi_p)$ we have that $\log(1/t)\in\Lambda(\varphi_p)$ if and only if
$0<p<1$. Hence, in view of Proposition \ref{variation-L}, $m_{\Lambda(\varphi_p)}$ has finite variation if and only if
$0<p<1$.
\end{exam}

Let $X, Y$ be r.i.\ spaces with $X\subseteq Y$, in which case there exists
$K>0$ such that $\|f\|_Y\le K \|f\|_X$ for $f\in X$. In particular,
$\|m_Y(A)\|_Y\le K\|m_X(A)\|_X$ for $A\in\mathcal{M}$. Hence, $m_Y$ has finite variation whenever $m_X$ does.
This observation, together with Proposition \ref{variation-L} and Example \ref{variation-L-rem}
establishes the following result.

\begin{cor}\label{condition var}
Let $X\not=\linf$ be a r.i.\ space.
Suppose that $\laf\subseteq X$ for some increasing, concave function $\vfi$ satisfying
$\vfi(0)=0$ and
$$
\int_0^1\log(1/t)\,\varphi'(t)\,dt<\infty,
$$
that is, $\log(1/t)\in\laf$. Then $\mx$ has finite variation.

In particular, since $\Lambda(\varphi_p)\subseteq M(\varphi_p)=L^p_{\textrm{exp}}$,
this is the case if $L^p_{\textrm{exp}}\subseteq X$ for some $0<p<1$.
\end{cor}

\begin{exam}\label{variation-L-rem2}
According to Corollary \ref{condition var},  $\mx$ has finite variation whenever
 $X$ is a Lorentz space $L^{p,q}$ on $[0,1]$ for $(p,q)\in(1,\infty)\times[1,\infty]$ or for $p=q=1$
(see \cite[Definition IV.4.1]{bennett-sharpley}), and whenever $X$ is an Orlicz space $L^\Phi$
satisfying $\Phi(t)\le e^{t{p}}$, $t\ge t_0$,  for some $p\in(0,1)$.
\end{exam}


\section{The Ces\`{a}ro space $\cx$}\label{3}


In \cite{curbera-ricker3} a study of optimal domains for  kernel
operators $Tf(x)=\int_0^1 f(y)K(x,y)\,dy$ was undertaken.
Although the conditions imposed  on the kernel $K(x,y)$ in \cite[\S3]{curbera-ricker3}
do not apply to the kernel $(x,y)\mapsto(1/x)\chi_{[0,x]}(y)$ generating the \ce operator,
a detailed analysis of the arguments given there shows that the only condition needed for
the results to remain valid for r.i.\
spaces $X\not=\linf$ is that the partial function $K_x\colon y\mapsto K(x,y)$ belongs to
$L^1$ for a.e.\ $x\in[0,1]$. The remaining  conditions were
aimed purely at guaranteeing that the vector measure associated with the kernel was $\sigma$-additive as
an $\linf$-valued measure which, in turn, was the way of ensuring  the $\sigma$-additivity
of the measure when interpreted as an $X$-valued measure (for every r.i.\ space $X\not=\linf$).
This last condition of $\sigma$-additivity is obtained, for the case when $T$ is the \ce operator,
by other means; see Theorem \ref{measure}(a). Accordingly,
from the results of \S3 of \cite{curbera-ricker3} we have the following facts.

\begin{pro}\label{3.1}
Let $X\not=\linf$ be a r.i.\ space.
The following assertions hold.
\begin{itemize}
\item[(a)] If $f\in \lmx$, then $f\in\cx$ and $\|f\|_{\lmx}= \|f\|_{\cx}$.
\item[(b)] If $X$ has a.c.\  norm, then $\cx$ has a.c.\  norm and $\cx=\lmx$.
\item[(c)] $[\mathcal{C},X_a]=\cx_a$.
\end{itemize}
Consequently, the following chain of inclusions holds
\begin{equation}\label{inclusions}
\lmxv\subseteq \lmx=L^1(m_{X_a})=[\mathcal{C},X_a]=\cx_a\subseteq\cx.
\end{equation}
\end{pro}

In this section we will study various  properties of $\cx$ and examine certain connections
between the spaces appearing in \eqref{inclusions}.

The containment $\lmx\subseteq\cx$ can be strict, as seen by the following result.

\begin{pro}\label{several}
Let $\vfi$ be an increasing,  concave function with $\vfi(0)=0$ and upper dilation index
$\delta_\vfi<1$. For the corresponding Marcinkiewicz space $M(\varphi)$
the containment $L^1(m_{M(\varphi)})\subseteq [\mathcal{C},M(\varphi)]$ is strict.
\end{pro}

\begin{proof}
The  a.c.-part of the space $M(\varphi)$ is
$$
\marf_a=M(\varphi)_0:=\left\{f:\lim_{t\to0} \frac{\varphi(t)}{t}\int_0^tf^*(s)\,ds=0\right\}.
$$
The condition $\delta_\vfi<1$ allows us to use the equivalent expression
for the norm in $\marf$ given by   \eqref{norm-marz}. The function $1/\varphi$ is
decreasing and so $(1/\vfi)^*=1/\vfi$. It follows that $\|1/\varphi\|_{\marf}\asymp 1$
and hence, $1/\varphi\in\marf$. On the other hand,
$$
\frac{\varphi(t)}{t}\int_0^t\bigg(\frac{1}{\varphi}\bigg)^*(s)\,ds\ge
\frac{\varphi(t)}{t}\frac{t}{\varphi(t)}=1,\quad t\in(0,1],
$$
showing that $1/\varphi\not\in\marf_0$. So, $1/\varphi\in\marf\setminus \marf_0$.

To verify that $\mathcal{C}(1/\varphi)\asymp 1/\varphi$  is equivalent to showing that
$(\vfi(t)/t)\int_0^tds/\vfi(s) \asymp 1$. Since $1/\varphi$ is decreasing (i.e., $(1/\vfi)^*=1/\vfi$), this
is equivalent to verifying $\|1/\varphi\|_{\marf}\asymp 1$, that is, to showing that $1/\varphi\in\marf$.
But, we have just proved that this is indeed the case,
due to the condition $\delta_\varphi<1$.  Hence, $\ces(1/\varphi)\in\marf\setminus \marf_0$ which implies that
$1/\varphi\in [\ces,\marf]\setminus [\ces,\marf_0]$. From \eqref{inclusions} we have that
$L^1(m_{M(\varphi)})=L^1(m_{M(\varphi)_0})=[\mathcal{C},M(\varphi)_0]$. Consequently,
$1/\varphi\in [\mathcal{C},M(\varphi)]\setminus L^1(m_{M(\varphi)})$.
\end{proof}


\medskip

We now establish two  properties of $\cx$ that were alluded to in the Introduction.

\begin{thm}\label{reflexive}
Let $X\not=\linf$ be any r.i.\ space.
\begin{itemize}
\item[(a)] The space $\lmx$ is not reflexive. Hence,
the Ces\`{a}ro space $[\mathcal{C},X]$ is not reflexive either.
\item[(b)] The Ces\`{a}ro space $[\mathcal{C},X]$  is not r.i. Moreover, neither is $\lmx=\cx_a$.
\end{itemize}
\end{thm}

\begin{proof}
(a) A general result concerning the $L^1$-space of a vector measure $m$ asserts that if  $m$
has  $\sigma$-finite variation and no atoms, then  $L^1(m)$ is not reflexive,
\cite[Remark p.3804]{curbera3},  \cite[Corollary 3.23(ii)]{okada-ricker-sanchez}.
Since this is the case for $\lmx$, which is a closed subspace of $\cx$,
it follows that $\cx$ is not reflexive either.

(b) Let $\varphi:=\fix$ be the fundamental function of $X$. Set
$f(t):=(-2\varphi^{-1/2})'(1-t)=\varphi'(1-t)/\varphi^{3/2}(1-t)$.
Since $f$ is an increasing function, $f^*(t)=\varphi'(t)/\varphi^{3/2}(t)$.
Direct computation shows that $\mathcal{C}f^*\equiv\infty$. Thus, $f^*\not\in[\mathcal{C},X]$.

On the other hand,
\begin{eqnarray*}
\mathcal{C}f(x)&=&\frac{1}{x}\int_0^x\frac{\varphi'(1-t)}{\varphi^{3/2}(1-t)}dt
= \frac{1}{x}\int_{1-x}^1\frac{\varphi'(s)}{\varphi^{3/2}(s)}\,ds
\\
&=&\frac{2}{x}
\left(\frac{1}{\varphi^{1/2}(1-x)}-\frac{1}{\varphi^{1/2}(1)}\right)
\\
&=& \frac{2}{\varphi^{1/2}(1)}\left(\frac{\varphi^{1/2}(1)-\varphi^{1/2}(1-x)}{x}\right)
\frac{1}{\varphi^{1/2}(1-x)}
\\
&:=&\frac{h(x)}{\varphi^{1/2}(1-x)}.
\end{eqnarray*}
Both of the functions $h$ and $1/h$ are bounded on $[0,1]$.
Accordingly, $\ces(f)$ is equivalent to $x\mapsto 1/\varphi^{1/2}(1-x)$.
So, $(\mathcal{C}f)^*(t)\asymp 1/\varphi^{1/2}(t)$. It follows that
$$
\|\mathcal{C}f\|_{\Lambda(\varphi)} \asymp \int_0^1 \frac{1}{\varphi^{1/2}(t)} \vfi'(t)\,dt
=2 \varphi^{1/2}(1)<\infty.
$$
Hence, $(\mathcal{C}f)^*\in \Lambda(\varphi)\subseteq X$, which implies that $f\in[\mathcal{C},X]$.
So, $\cx$ is not r.i.

According to \eqref{inclusions} we have $L^1(m_{X_a})=[\ces,X_a]$. Since $[\ces,X_a]$ fails to be r.i., so does
$L^1(m_{X_a})$.  But, $L^1(m_{X_a})=\lmx$. Accordingly, the closed subspace
$\lmx$ of $\cx$ is never r.i.
\end{proof}


\begin{rem}\label{wsc}
(a) A reasonable `substitute' for reflexivity is weak sequential completeness. If $X$ is
weakly sequentially complete, then $\cx$ is also weakly sequentially complete. Indeed,
the weak sequential completeness of $X$ implies that of $\lmx$, \cite[Corollary to Theorem 3]{curbera1}.
But, $\lmx=\cx$; see Proposition \ref{3.1}(a).

(b) Some further examples of vector measures $m$ for which the spaces $L^1(m)$ are known
not to be r.i.\ arise from Rademacher functions,
\cite[Theorem 1]{curbera4}, and from fractional integrals, \cite[Example 5.15(b)]{curbera-ricker1}.
\end{rem}


\medskip

We now address the question of when $\cx$ is order isomorphic to an \textit{AL-space}, that is, to a Banach lattice
in which the norm is additive over disjoint functions.
In this regard, the space $X=\ele$ exhibits a particular feature, namely, that
\begin{equation}\label{caseL1}
[\ces,L^1]=L^1(m_{L^1})=L^1(|m_{L^1}|)=L^1(\log(1/t) .
\end{equation}
We point out that not only do the three spaces $[\ces,L^1]$, $L^1(m_{L^1})$ and $L^1(|m_{L^1}|)$
coincide, but  that $[\ces,L^1]$ is also an AL-space.


\begin{pro}\label{AL}
Let $X\not=\linf$ be a r.i.\ space. The following conditions are equivalent.
\begin{itemize}
\item[(a)] The space $\cx$ is order isomorphic to an AL-space.
\item[(b)] The spaces $\lmx$ and $\lmxv$ are order isomorphic
via the natural inclusion (this latter condition is written as
$\lmx\simeq\lmxv$).
\item[(c)] The function $y\mapsto \|F_y\|_X$, $y\in[0,1]$, belongs to the associate space $\cx'$.
\end{itemize}

If any one of these conditions holds, then
$$
\cx=\lmx\simeq\lmxv.
$$
\end{pro}

\begin{proof}
(a) $\Rightarrow$ (b) If $\cx$ is order isomorphic to an AL-space, then it is a.c.,
\cite[Theorem 1.a.5 and Proposition 1.a.7]{lindenstrauss-tzafriri}.
Hence, by \eqref{inclusions} we have that  $\cx=\cx_a=\lmx$ and
so $\lmx$ is order isomorphic to an AL-space. This last condition implies that
$\lmx$  is order isomorphic (via the natural inclusion) to $\lmxv$;
see Proposition 2 of \cite{curbera2} and its proof.

(b) $\Rightarrow$ (a) Suppose that $\lmx\simeq\lmxv$.
According to \eqref{inclusions} we have $\cx_a=\lmx$ and so
$\cx_a\simeq\lmxv$. Since $\lmxv$ is weakly sequentially complete,
it follows that $\cx_a$ has the Fatou property
and hence, that $\cx_a=(\cx_a)''$. Since $\cx_a\not=\{0\}$,  we have
$(\cx_a)'=\cx'$ and hence,
$(\cx_a)''=\cx''$. Thus, $\cx_a=\cx''$ which, together with
the chain of inclusions $\cx_a\subseteq \cx \subseteq \cx''$, yields
 $\cx=\cx_a$. Accordingly, $\cx\simeq\lmxv$ and this last space is an AL-space.

(b) $\Leftrightarrow$ (c) Due to \eqref{inclusions} we have $\lmx=[\ces,X_a]$.
Hence, the condition $\lmx\simeq\lmxv$ is equivalent to $[\ces,X_a]\simeq\lmxv$.
This, in turn, is equivalent to
the requirement
$$
\int_0^1|f(y)| \cdot\|F_y\|_{X}\,dy<\infty, \quad  f\in [\ces,X_a],
$$
which is precisely the condition that the function $y\mapsto\|F_y\|_{X}$
belongs to  the associate space $[\ces,X_a]'=\cx'$.
\end{proof}


In the sequel we will repeatedly use  the fact that $\mathcal{C}\colon X\to X$ (necessarily boundedly)
if and only if   $X\subseteq \cx$. For
r.i.\ spaces $X$ this corresponds precisely to  the upper Boyd index $\overline{\alpha}_X$ of $X$
satisfying  $\overline{\alpha}_X<1$; see \cite[II.6.7, Theorem 6.6]{krein-petunin-semenov}
or \cite[Remark 5.13]{maligranda}. Note that the proof given
in \cite[Theorem III.5.15]{bennett-sharpley} uses the Fatou property of $X$.
Observe that if $\mathcal{C}\colon X\to X$, then also $\ces\colon X_a\to X_a$.


\begin{thm}\label{L1}
Let $\vfi$ be an increasing,  concave function with $\vfi(0)=0$ and having non-trivial dilation indices
$0<\gamma_\vfi\le\delta_\vfi<1$.
\begin{itemize}
\item[(a)] For $X=\Lambda(\vfi)$ the B.f.s.\ $\cx$ is order isomorphic to an AL-space.
\item[(b)] For $X=\marf$ the B.f.s.\ $\cx$ is not order isomorphic to an AL-space.
\end{itemize}
\end{thm}

\begin{proof}
Via Proposition \ref{AL}, we need to decide whether or not    $\lmx\simeq\lmxv$.

In \cite[Corollary 13]{lesnik-maligranda-1} Les\u{n}ik and Maligranda
identify the associate space of $\cx$ in the case
when $X$ has the Fatou property and both $\ces,\ces^*$ act boundedly on $X$.
Here $f\mapsto \ces^*(f)(x):=\int_x^1\frac{f(t)}{t}\,dt$,
$x\in[0,1]$, for any a.e.\ finite measurable function $f$
(denoted by $f\in L^0$) for which it is meaningfully defined, is the \textit{Copson operator}.
Then
\begin{equation}\label{c*}
\cx'= \left(X'\Big(\frac{1}{1-x}\Big)\right)^{\widetilde{}}
=\bigg\{f: y\mapsto\frac{\tilde{f}(y)}{1-y}\in X'\bigg\},
\end{equation}
where $\tilde{f}$ is the decreasing majorant of  $f$,
defined by $\tilde{f}(y):=\sup_{x\geqslant y}|f(x)|$ and, for a weight function
$0<w$ on $[0,1]$ and a B.f.s.\ Y, we set $Y(w):=\{h:wh\in Y\}$ and $\tilde{Y}:=\{g:\tilde{g}\in Y\}$.

(a) The identification \eqref{c*}  applies to $X=\laf$ as $X$ possesses  the
Fatou property and because $\underline{\alpha}_X=\gamma_\vfi$ and
$\overline{\alpha}_X=\delta_\vfi$, together with  the given index assumptions, imply that
$0<\underline{\alpha}_X \le \overline{\alpha}_X<1$ which, in turn, guarantees  that
$\ces,\ces^*\colon \laf\to \laf$ boundedly.

Since $\Lambda(\varphi)'=M(\psi)$,  for
$\psi(t):=t/\vfi(t)$, we have from  \eqref{c*} that
\begin{equation*}
[\ces,\laf]'
= \left(M(\psi)\Big(\frac{1}{1-y}\Big)\right)^{\widetilde{}}
=\bigg\{f:\sup_{0<t\le1}\frac{1}{\vfi(t)}
\int_0^t\Big(\frac{\tilde{f}(y)}{1-y}\Big)^*(s)\,ds<\infty\bigg\}.
\end{equation*}
The condition $0<\gamma_\vfi$ implies that $\delta_{\psi}<1$
which allows us, via \eqref{norm-marz}, to simplify the  previous description to
\begin{equation*}\label{dual-simple}
[\ces,\laf]'
=\bigg\{f:\sup_{0<t\le1}\frac{t}{\vfi(t)}
\Big(\frac{\tilde{f}(y)}{1-y}\Big)^*(t)<\infty\bigg\}.
\end{equation*}

We  need to verify that $y\mapsto \|F_y\|_{\Lambda(\varphi)}\in [\ces,\laf]'$;
see Proposition \ref{AL}. From \eqref{normL} it follows that
$$
\|F_y\|_{\Lambda(\varphi)}= \int_0^{1-y} \frac{\varphi'(s)}{y+s}\,ds.
$$
This function is decreasing (as a function of its variable $y$),
so it coincides with its decreasing majorant,
that is, $(\|F_y\|_{\Lambda(\varphi)})^{\widetilde{}}=\|F_y\|_{\Lambda(\varphi)}$.
Moreover, for $0<y\le1$, we have
\begin{eqnarray*}
\frac{\|F_y\|_{\Lambda(\varphi)}}{1-y}
&\le&
2 \chi_{[0,1/2]}(y)\int_0^{1} \frac{\varphi'(s)}{y+s}\,ds
+
\chi_{[1/2,1]}(y)\frac{2}{1-y}\int_0^{1-y} \varphi'(s)\,ds
\\ &\le &
2\int_0^{1} \frac{\varphi'(s)}{y+s}\,ds +2  \frac{\varphi(1-y)}{1-y}
\\ &:= &
g(y)+h(y).
\end{eqnarray*}
In the latter term, $g$ is decreasing and $h$  is increasing
due to the quasi-concavity of $\vfi$  (which implies that $\vfi(t)/t$ is decreasing), i.e.,
$g^*=g$ and $h^*(t)=h(1-t)$.  Using the property $(g+h)^*(t)\le g^*(t/2)+h^*(t/2)$
(see (2.23) in \cite[Ch.II \S2, p.67]{krein-petunin-semenov}),
it follows that
$$
\bigg(\frac{\|F_y\|_{\Lambda(\varphi)}}{1-y}\bigg)^*(t)\le
g\Big(\frac{t}{2}\Big) + h\Big(1-\frac{t}{2}\Big) =
2 \int_0^{1} \frac{\varphi'(s)}{\frac t2+s}\,ds +2 \frac{\varphi(t/2)}{t/2}.
$$
Accordingly,
$$
\sup_{0<t\le1}\frac{t}{\vfi(t)}\bigg(\frac{\|F_y\|_{\Lambda(\varphi)}}{1-y}\bigg)^*(t)
\le
2
\sup_{0<t\le1}\frac{t}{\vfi(t)} \int_0^{1} \frac{\varphi'(s)}{\frac t2+s}\,ds
+
4\sup_{0<t\le1}\frac{t}{\vfi(t)}\frac{\varphi(t/2)}{t } .
$$
The last term in the right-side is bounded (as $\vfi$  increasing implies $\vfi(t/2)/\vfi(t)\le1$)
and so we concentrate on the first term.
Due to the quasi-concavity of $\vfi$ we have $t\vfi'(t)\le \vfi(t)$.
This, together with a change of variables
yields, for $t\in(0,1]$, that
\begin{eqnarray*}
\frac{t}{\vfi(t)} \int_0^{1} \frac{\varphi'(s)}{\frac t2+s}\,ds
\le
\int_0^{1} \frac{\varphi(s)}{\vfi(t)}\frac{t}{s(\frac t2+s)} \,ds
\le
\int_0^{1/t} \frac{\varphi(tu)}{\vfi(t)}\frac{2du}{u(1+u)} =I_t.
\end{eqnarray*}
The conditions $0<\gamma_\varphi\le \delta_\varphi<1$ imply
that there exist $\alpha, \beta\in(0,1)$, and
$u_0, u_1$ with $0<u_0< 1<u_1<\infty$ such that
\begin{equation*}
\frac{\vfi(tu)}{\vfi(t)}\le u^\alpha,\quad 0<u<u_0,
\qquad \frac{\vfi(tu)}{\vfi(t)}\le u^\beta,\quad u>u_1,
\end{equation*}
\cite[pp.53-56]{krein-petunin-semenov}. Since  $\vfi(tu)/\vfi(t) \le \max\{1,u_1\}=u_1$,
for $u_0<u<u_1$ (via the quasi-concavity of $\vfi$), it follows that
\begin{equation*}
I_t \le  \int_0^{u_0} \frac{2u^\alpha du}{u(1+u)}
+
\int_{u_0}^{u_1}  u_1\frac{2du}{u(1+u)}
+
\int_{u_1}^{\infty} \frac{2u^\beta du}{u(1+u)},
\end{equation*}
which is finite as $0<\alpha,\beta<1$.
Thus,  $\|F_t\|_{\Lambda(\varphi)}\in [\ces,\laf]'$ and  so $L^1(m_{\laf})\simeq L^1(|m_{\laf}|)$.
Hence, $[\ces,\laf]$ is order isomorphic to an AL-space.

\medskip

(b) For $X=\marf$ the identification \eqref{c*}  can again be applied,
for the same reasons that it was applied in the case of $\laf$; see part (a).
In particular, both $\ces,\ces^*\colon \marf\to \marf$ boundedly.

Since $\marf'=\Lambda(\psi)$, for  $\psi(t):=t/\vfi(t)$, we have
from  \eqref{c*} that
\begin{equation*}
[\ces,\marf]'
=
\left(\Lambda(\psi)\Big(\frac{1}{1-y}\Big)\right)^{\widetilde{}}
=
\bigg\{f:\int_0^1 \Big(\frac{\tilde{f}(y)}{1-y}\Big)^*(t)\, \psi'(t)\,dt<\infty\bigg\}.
\end{equation*}

We  need to verify that $y\mapsto \|F_y\|_{\marf}\not\in [\ces,\marf]'$; see Proposition \ref{AL}.
Since the upper dilation index of
$\varphi$ satisfies $\delta_\varphi<1$,
we can use the equivalent expression \eqref{norm-marz} for the norm in
$\marf$  to obtain from \eqref{density*} that
$$
\|F_y\|_{\marf}\asymp \sup_{0\le s\le 1-y} \frac{\vfi(s)}{s+y} .
$$
This function is decreasing (as a function of its variable $y$) and  so it coincides with its decreasing majorant,
$(\|F_y\|_{\marf})^{\widetilde{}}=\|F_y\|_{\marf}$. Moreover, for each $y\in[0,1]$, we have
$$
\|F_y\|_{\marf}\asymp\sup_{0\le s\le 1-y} \frac{\vfi(s)}{s+y} \ge \vfi(1-y) ,
$$
and hence, modulo a positive constant,
$$
\frac{\|F_y\|_{\marf}}{1-y} \ge \frac{\vfi(1-y)}{1-y}  .
$$
Since $\vfi$ is quasi-concave, $\vfi(t)/t$ is decreasing and so
$\big(\frac{\vfi(1-t)}{1-y}\big)^*(t)=\vfi(t)/t$, i.e.,
$$
\left(\frac{\|F_y\|_{\marf}}{1-y}\right)^*(t) \ge \frac{\vfi(t)}{t} = \frac{1}{\psi(t)}.
$$
Accordingly, modulo a positive constant, we have
$$
\int_0^1 \bigg(\frac{(\|F_y\|_{\marf})^{\widetilde{}}}{1-y}\;\bigg)^*(t)\, \psi'(t)\,dt
\ge
\int_0^1 \frac{\psi'(t)}{\psi(t)}\,dt =\infty.
$$
Hence, $\|F_y\|_{\Lambda(\varphi)}\not\in [\ces,\marf]'$ and  so $L^1(m_{\marf})\not=L^1(|m_{\marf}|)$.
Consequently, $[\ces,\marf]$ is not order isomorphic to an AL-space.
\end{proof}

A precise description of  when $\cx$ is a weighted $L^1$-space
(in particular, an AL-space) can be deduced from  \cite[Theorem 3.3]{schep}.

\begin{rem}
(a) The argument at the beginning of the proof of $(a)\Rightarrow (b)$ in Proposition \ref{AL}
shows that  also  $[\ces,\marf_0]$ is not order isomorphic to an AL-space.

(b) If $\cx$ is order isomorphic to an AL-space, then Proposition \ref{AL} implies that $\cx=\lmx\simeq\lmxv$.
Thus, $\chi_{[0,1]}\in\lmxv$ and so $\mx$ has finite variation.
Hence, whenever $\mx$ has infinite variation (e.g. $X=L^p_{\textrm{exp}}$,
$p\ge1$, or if $\log(1/t)\not\in\laf$),
then $\cx$ cannot be order isomorphic to an AL-space.

(c) Further examples of when $\cx$ fails to be order isomorphic to an AL-space
occur in Proposition \ref{cc lmx} below.
\end{rem}


\medskip

The final results of this section address the question of when is $X$  contained
in  $\lmx$ or in $\lmxv$. In the first case, we have  the integral representation for $\ces\colon X\to X$
as given in \eqref{representationBDS}  via the Bartle-Dunford-Schwartz integral. In the latter case,
the representation for $\ces\colon X\to X$ is via the Bochner integral as given by \eqref{representation} and
\eqref{density}.

\begin{rem}\label{3.6}
(a) Let $X\not=\linf$ be a r.i.\ space such that $\overline{\alpha}_X<1$.
Then each of the containments $X\subseteq\cx$ and $X_a\subseteq L^1(\mx)$ is proper.
Indeed, since $\overline{\alpha}_X<1$, we have $X\subseteq\cx$, where $X$ is r.i.\ and $\cx$ is not; see
Theorem \ref{reflexive}(b). Thus,  $X=\cx$ is impossible.

Applying the previous  argument to $X_a$ (in place of $X$) shows that
$X_a\subseteq[\ces,X_a]$ properly. But, $[\ces,X_a]=\lmx$; see \eqref{inclusions}.

If, in addition, $X$ has a.c.\ norm, then $X_a=X$ and so $X\subseteq\lmx=\cx$ properly.

(b) Unlike for the containment $X_a\subseteq\lmx$, it is not true in general
(with $\overline{\alpha}_X<1$) that $X\subseteq\lmx$.
Indeed, for  $X=L^{p,\infty}$, $1<p<\infty$, we have
$$
X_a=L^{p,\infty}_0=\left\{f:\lim_{t\to0} t^{-1/q}\int_0^tf^*(s)\,ds=0\right\},\quad \frac1p+\frac1q=1.
$$
Since $\overline{\alpha}_X=1/p<1$, it follows from part (a)  that
$L^{p,\infty}_0\subseteq L^1(m_{L^{p,\infty}})$ properly.
To see that $L^{p,\infty}\not\subseteq L^1(m_{L^{p,\infty}})$ we consider,
as in the proof of Proposition \ref{several}, the decreasing function $x^{-1/p}\in L^{p,\infty}\setminus L_0^{p,\infty}$.
Since $\mathcal{C}(x^{-1/p})=qx^{-1/p}$, it follows  that
$x^{-1/p}\not\in  [\mathcal{C},L_0^{p,\infty}]$.
From \eqref{inclusions} we have $[\mathcal{C},L_0^{p,\infty}]=L^1(m_{L^{p,\infty}})$.
Accordingly, $x^{-1/p}\in L^{p,\infty}$ and
$x^{-1/p}\not\in L^1(m_{L^{p,\infty}})$.
\end{rem}


\medskip

Let $X=\Lambda(\vfi)$ satisfy $0<\gamma_\vfi\le\delta_\vfi<1$.
It follows from Proposition \ref{L1}(a) that
$\cx\simeq\lmxv$. Since $\overline{\alpha}_X=\delta_\vfi<1$,
we also have $X\subseteq\cx$ and hence,  $X\subseteq\lmxv$.
On the other hand, for $X=L^1$ we have the contrary situation that
$L^1\not\subseteq L^1(\log(1/t))=L^1(m_{L^1})$; see \eqref{caseL1}.
Of course, here $\overline{\alpha}_X=\delta_\vfi=1$.
A similar situation occurs for $X=L^p$, $1<p<\infty$,
namely $L^p\not\subseteq L^1(|m_{L^p}|)$, \cite[Theorem 1.1(ii)]{ricker}.
The following result exhibits additional facts concerning whether or not we have
$X\subseteq\lmxv$.

\begin{pro}\label{marf x and lmxv}
Let $X\not=\linf$ be a r.i.\ space.
\begin{itemize}
\item[(a)] It is always the case that $L^1\not\subseteq\lmxv$.
\item[(b)] Suppose  that $X\subseteq\lmxv$. Then the containment is necessarily proper.
\item[(c)] The containment $X\subseteq\lmxv$ holds if and only if the function
$$
y\mapsto \|F_y\|_X=\Big\|t\mapsto\frac{1}{t+y}\chi_{[0,1-y]}\Big\|_X, \quad y\in(0,1],
$$
belongs to the associate space $X'$ of $X$.
\item[(d)] For $X=\marf$, it is the case that  $\marf\not\subseteq L^1(|m_{\marf}|)$.
\end{itemize}
\end{pro}

\begin{proof}
(a) If $L^1\subseteq\lmxv$ holds, then $\int_0^1|f(y)|\cdot\|F_y\|_X\,dy<\infty$ for all $f\in L^1$ and so
$\sup_{0<y\le1}\|F_y\|_X<\infty$. However, this is impossible since, for each $y\in(0,1]$, we have
$$
\|F_y\|_X\ge \|F_y\|_{L^1}=\int_0^1\frac{1}{x}\chi_{[y,1]}(x)\,dx=\log(1/y).
$$

(b) If $X\simeq\lmxv$ holds, then $X$ is a r.i.\ space which is order isomorphic to an AL-space.
Then, for some constants $C_1,C_2>0$, we have $C_1\|f\|_{\lmxv}\le \|f\|_X\le C_2 \|f\|_{\lmxv}$, $f\in X$.
So, for $0\le s<t\le1$, we have
\begin{eqnarray*}
\|\chi_{[0,t]}\|_X &\le& C_2 \|\chi_{[0,t]}\|_{\lmxv} =C_2\big(\|\chi_{[0,s]}\|_{\lmxv}+\|\chi_{[s,t]}\|_{\lmxv}\big)
\\ &\le&
\frac{C_2}{C_1}\big(\|\chi_{[0,s]}\|_X+\|\chi_{[s,t]}\|_X\big)
=
\frac{C_2}{C_1}\big(\|\chi_{[0,s]}\|_X+\|\chi_{[0,t-s]}\|_X\big)
.
\end{eqnarray*}
In a similar way we can obtain the corresponding lower bound. It follows that the fundamental function $\fix$
satisfies $\fix(s+t)\asymp \fix(s)+\fix(t)$ for $s,t,s+t\in[0,1]$. This, together with the continuity
of $\fix$ on $[0,1]$ and $\fix(0)=0$, implies that $\fix(ta)\asymp t\fix(a)$ for $t, a, ta\in[0,1]$.
Hence, $\fix(t)\asymp t$ for $t\in[0,1]$, which implies that $X$ is order isomorphic to $L^1$.
But, this contradicts part (a).

(c) Note that $X\subseteq \lmxv$ if and only if $f\in\lmxv$ for all $f\in X$, that is (via Theorem \ref{measure}),
$$
\int_0^1|f(y)| \cdot\|F_y\|_X\,dy<\infty, \quad  f\in X,
$$
which corresponds to the function $y\mapsto\|F_y\|_X$ belonging to the space $X'$.
Since $X$ is r.i., it follows from \eqref{density*} that this is equivalent to the function
$$
y\mapsto \|(F_y)^*\|_X=\Big\|t\mapsto\frac{1}{t+y}\chi_{[0,1-y]}\Big\|_X, \quad y\in(0,1],
$$
belonging to $X'$.

(d) Applying part (c)  to $X=\marf$ we need to show
that $y\mapsto\|F_y\|_{\marf}$ does not
belong to $\marf'=\Lambda(\psi)$, for $\psi(t):=t/\vfi(t)$.

The function $y\mapsto\|F_y\|_{\marf}$ can be estimated below,
using  \eqref{density*}, for the values $0\le y\le 1/2$ (in which case $y\le 1-y$), namely
\begin{eqnarray*}
\|F_y\|_{\marf} &=& \sup_{0<t\le1} \frac{\vfi(t)}{t} \int_0^t (F_y)^*(s)\,ds
\\ &\ge&
\sup_{0<t\le1} \vfi(t)(F_y)^*(t)
=
\sup_{0<t\le1-y} \frac{\vfi(t)}{y+t}
\ge
\frac{\vfi(y)}{2y}.
\end{eqnarray*}
Hence, we have that
$$
(\|F_y\|_{\marf})^*(t)\ge \frac{\vfi(t)}{2t}=\frac12 \frac{1}{\psi(t)},\quad 0<t\le \frac12.
$$

Consequently,
\begin{equation*}
\big\|y\mapsto\|F_y\|_{\marf}\;\big\|_{\Lambda(\psi)}
=
\int_0^{1}(\|F_t\|_{\marf})^*(t) \psi'(t)\,dt
\ge
\frac12 \int_0^{1/2}  \frac{\psi'(t)}{\psi(t)}\,dt
= \infty.
\end{equation*}
\end{proof}

\begin{rem}
The proof of Proposition \ref{marf x and lmxv}(d) shows that the result also applies
to the a.c.\ part $\marf_0$ of $\marf$. More generally, for a r.i.\ space $X\not=\linf$ we have
$X_a\not\subset L^1(|m_{X_a}|)$  if and only if $X\not\subset\lmxv$ since
$X_a$ and $X$ have the same norm and $X_a'=X'$.
\end{rem}

Regarding the separability of $\cx$, it is known that the B.f.s.'
$Ces_p([0,1])$, $1<p<\infty$, are \textit{separable},
\cite[Theorem 1]{astashkin-maligranda1}, \cite[Theorem 3.1(b)]{astashkin-maligranda2}.
As pointed out in \cite[p.18]{astashkin-maligranda2}, this
is due to the fact that $Ces_p([0,1])$, which coincides with $[\ces,L^p]_a=[\ces,L^p]$, contains
$\linf$ and has a.c.\ norm. More generally, since $\linf\subseteq X_a$ for any r.i.\ space $X\not=\linf$
and $\ces\colon\linf\to\linf$, we necessarily have that $\linf\subseteq[\ces,X_a]=[\ces,X]_a$
(cf. Proposition \ref{3.1}) and hence, $[\ces,X]_a$ is separable by the a.c.\ of its norm.
In particular, if $X$ itself has a.c.\ norm, then $\cx$ is separable; see \eqref{inclusions}. Since
the $\sigma$-algebra $\mathcal{M}$ is $\lambda$-essentially countably generated and
$\lmx=\cx_a$, via \eqref{inclusions}, this also follows from a general result on the separability
of $L^1(m)$, \cite[Proposition 2]{ricker2}.


\section{The \ce operator acting on $\cx$}


It is known that the operator $\ces\colon \elp\to\elp$, for $1<p<\infty$,
is  not compact,  \cite[p.28]{leibowitz}.
Actually, this is a rather general feature.

\begin{pro}\label{4.1}
Let $X\not=\linf$ be a r.i.\ space satisfying $\overline{\alpha}_X<1$. Then the continuous
operator $\ces\colon X\to X$ is not compact.
\end{pro}

\begin{proof}
For each $\alpha\ge0$, direct calculation  shows that the continuous
function $x^\alpha$ (on $[0,1]$) satisfies $\ces(x^\alpha)=x^\alpha/(\alpha+1)$ and so $1/(\alpha+1)$ is an
eigenvalue of $\ces$. Accordingly, the interval $(0,1]$ is contained in the spectrum of $\ces$ and so $\ces$ cannot
be compact.
\end{proof}

Since the operator $\ces\colon X\to X$, whenever it is available,
factorizes through $\lmx$ via $\imx\colon\lmx\to X$, it follows
that also $\imx$ is not compact. By the same argument also
$\ces\colon\cx\to X$ fails to be compact.
Actually, the requirement that $\ces\colon X\to X$ is unnecessary.

\begin{pro}\label{compact}
Let $X\not=\linf$ be any r.i.\ space. Then  the operator $\ces\colon\cx\to X$ is not compact.
\end{pro}

\begin{proof}
According to  \cite[Theorem 4]{okada-ricker-rpiazza1}, the bounded variation of $\mx$ is
a necessary condition for  $\imx\colon\lmx\to X$ to be compact.
Thus, if $\mx$ has infinite variation,  then $\imx\colon\lmx\to X$ is not compact.
Since the restriction of $\ces\colon\cx\to X$ to the closed subspace $\lmx$ is $\imx$, also
$\ces\colon\cx\to X$ fails to be compact.

Suppose now that $\mx$ has finite variation. Then a further condition is necessary for $\imx\colon\lmx\to X$ to be compact:
the existence of a Bochner integrable density, in our case the function $F\colon y\mapsto F_y$, with the property that
the set $\mathcal{B}:=\{G(y):=F_y/\|F_y\|_X, 0\le y\le1\}$ is relatively compact in $X$,
\cite[Theorem 1]{okada-ricker-rpiazza1}. So, assume then that this last condition holds.
Choose a sequence $\{y_n\}\subseteq[0,1]$ which increases to 1.
Since $\{G_{y_n}\}\subseteq \mathcal{B}$, there is a subsequence, again denoted  by
 $\{G_{y_n}\}$ for convenience, which converges in $X$. Let $\psi\in X$ be the limit of $\{G_{y_n}\}$.
By passing to a subsequence, if  necessary, we can assume that $G_{y_n}(x)\to \psi(x)$ for a.e.\ $x\in[0,1]$.
Recall that  $F_y$ is given by $F_y(x)=(1/x)\chi_{[y,1]}(x)$; see \eqref{density}.
Thus, as $\{y_n\}$ increases to $1$ we have $F_{y_n}(x)\to0$ for a.e. $x\in[0,1]$. The same property occurs also for
$\{G_{y_n}\}$. As a consequence, $\psi=0$ a.e. This contradicts  $\|\psi\|_X=1$ as $\|G_{y_n}\|_X=1$ for all $n\ge1$.
\end{proof}


A useful substitution for compactness
is the  \textit{complete continuity} of an  operator, that is, one
which maps weakly convergent sequences to norm convergent sequences.
In view of the Eberlein-\u{S}mulian Theorem, this is equivalent to mapping relatively weakly compact
sets to relatively norm compact sets.
For the particular case of the \ce operator, due to the fact that the vector measure $\mx$ has relatively compact range and
$\sigma$-finite variation (cf. Theorem \ref{measure}(c), (d)) it is the case that the (restricted) integration operator
$\imx\colon\lmxv\to X$ is always completely continuous, \cite[Proposition 3.56]{okada-ricker-sanchez}.
This fact will have important consequences.

The following result should be compared with Proposition \ref{4.1}.

\begin{pro}\label{cc lmxv}
Let $X\not=\linf$ be a r.i.\ space such that the  function $y\mapsto\|F_y\|_X$ belongs to  $X'$.
Then $\ces\colon X\to X$ is  completely continuous.

In particular, this occurs for
$X=\Lambda(\vfi)$ if $\varphi$ satisfies $0<\gamma_\varphi\le \delta_\varphi<1$.
\end{pro}

\begin{proof}
By Proposition \ref{marf x and lmxv}(c) the function $y\mapsto\|F_y\|_X$ belonging to  $X'$
implies that $X\subseteq\lmxv$. According to \eqref{inclusions} we have $X\subseteq\cx$ and so the operator
$\ces\colon X\to X$ is  continuous. Moreover, it can be factorized via the continuous inclusion
$X\subseteq\lmxv$ and the restricted integration operator $\imx\colon\lmxv\to X$.
But, as noted above, $\imx\colon\lmxv\to X$ is necessarily completely continuous.
The ideal property of completely continuous operators then implies
that $\ces\colon X\to X$ is also completely continuous.

The particular case of $X=\Lambda(\vfi)$ with $0<\gamma_\varphi\le \delta_\varphi<1$ follows from
Proposition \ref{AL} and Theorem \ref{L1}(a).
\end{proof}

\begin{rem}
Any r.i.\ space $X$ for which the function $y\mapsto\|F_y\|_X$ belongs to $X'$ cannot be reflexive. For, if so, then
$\ces\colon X\to X$ is a completely continuous operator defined
on a reflexive Banach space and hence,  it is necessarily compact (which contradicts Proposition \ref{4.1}).
For $X=L^p$, $1<p<\infty$, this was shown explicitly in (the proof of)
Theorem 1.1(i) in \cite{ricker}.
\end{rem}


We deduce some further consequences from the complete continuity of the restricted integration operator
$\imx\colon\lmxv\to X$.

\begin{pro}\label{cc lmx}
Let $X\not=\linf$ be a r.i.\ space.
\begin{itemize}
\item[(a)]If $\cx$ is order isomorphic to an AL-space, then $\ces\colon\cx\to X$ is completely continuous.
\item[(b)] Let $X$ be  reflexive and satisfy $\overline{\alpha}_X<1$. Then
$\ces\colon\cx\to X$ is not completely continuous. In particular,
$\cx$  cannot be order isomorphic to an AL-space.
\item[(c)] Suppose that $X$  does not contain a copy of $\ell^1$. If
$\ces\colon\cx\to X$ is  completely continuous,
then  $\cx$ is  order isomorphic to an AL-space.
\end{itemize}
\end{pro}

\begin{proof}
(a) From Proposition \ref{AL} we have $\cx\simeq\lmxv$ which, together with
the operator $\imx\colon\lmxv\to X$ being completely continuous, establishes the claim.

(b) From $\overline{\alpha}_X<1$ we have $\ces\colon X\to X$ and so $X\subseteq\cx$.
Then, $\ces\colon X\to X$ can be factorized via $\ces\colon\cx\to X$.
Suppose that $\ces\colon\cx\to X$ is completely continuous. Then also $\ces\colon X\to X$
is completely continuous. Since $X$ is reflexive, we conclude that $\ces\colon X\to X$
is compact, which is a contradiction to Proposition \ref{4.1}.

Suppose now that $\cx$ is order isomorphic to an AL-space. Then part (a) implies that
$\ces\colon\cx\to X$ is completely continuous. But, we have just proved that this is not possible.

(c) Suppose that $\cx$ is not order isomorphic to an AL-space. Then it follows
from Proposition \ref{AL} that $\lmxv\not=\lmx$. On the other hand, the complete continuity of $\ces\colon\cx\to X$
implies (via factorization through $\lmx$) that
$\imx\colon\lmx\to X$ is also completely continuous. Combining  Corollary 1.4 of
\cite{calabuig-etal} with Proposition 1.1  of \cite{okada-ricker-rpiazza2},
it follows  that $\lmxv\simeq\lmx$. Contradiction!
\end{proof}

\begin{rem}
For an example of a reflexive r.i.\ space with $\overline{\alpha}_X=1$
we refer to \cite[Example 12, p.29]{maligranda}.
\end{rem}

Recall that a Banach space $X$ has the \textit{Dunford-Pettis property} if
every Banach-space-valued, weakly compact linear operator defined on $X$ is completely
continuous. The classical example of a space with this property is $L^1$.
In Theorem \ref{L1}(a) it was established, for certain Lorentz spaces $\laf$, that
$[\ces,\laf]\simeq L^1(|m_{\laf}|)$ with $|m_{\laf}|$ a finite, non-atomic measure.
Hence, $[\ces,\laf]$ has the Dunford-Pettis property in this case. However,
as noted in the Introduction, $Ces_p=[\ces,L^p]$, $1<p<\infty$, fails the
Dunford-Pettis property,  \cite[\S6, Corollary 1]{astashkin-maligranda1}.
The proof of this given in \cite{astashkin-maligranda1} relies on some  results concerning
certain Banach space properties particular to $Ces_p$. The following extension of this result
is established via the methods of vector measures.

\begin{pro}\label{4.7}
Let $X$ be any reflexive r.i.\ space with $\overline{\alpha}_X<1$.
Then, $\cx$ fails the Dunford-Pettis property.
\end{pro}

\begin{proof}
Since $X$ has a.c.\ norm, we have $\lmx=\cx$ (cf. \eqref{inclusions}) and hence,
because of $\overline{\alpha}_X<1$, it follows that $\ces\colon X\to X$ and so $X\subseteq\cx=\lmx$.
Suppose that $\cx$ has the Dunford-Pettis property. Then the weakly compact operator $\imx\colon\lmx\to X$
(recall that $X$ is reflexive) is necessarily completely continuous. Since $\ces\colon X\to X$ is the composition
of $\imx\colon\lmx\to X$ and the natural inclusion of $X$ into $\lmx$, it follows that
$\ces\colon X\to X$ is completely continuous. The reflexivity of $X$ then ensures that $\ces\colon X\to X$
is actually compact. But, this contradicts Proposition \ref{4.1}. Accordingly, $\cx$ fails the
Dunford-Pettis property.
\end{proof}


\section{The Fatou property for $\cx$}


In \cite[Theorem 1(d)]{lesnik-maligranda-1} it was noted that if $X$
has the Fatou property, then also $\cx$ has the Fatou property.
As explained in the beginning of \S3, the results on  optimal domains for  kernel operators
given in \cite[\S3]{curbera-ricker3} also apply to the kernel
generating the \ce operator (and to many other operators).
In \cite{curbera-ricker3}, a fine analysis of the Fatou property was undertaken.
Proposition \ref{3.1} above presents a partial view of the relations between
the various function spaces involved. The complete picture of the results
in \cite[\S3]{curbera-ricker3}  is presented below. It involves the
space $\wlmx$ consisting of all the functions which are \textit{weakly integrable} with respect to
the vector measure $\mx$, that is, of all  measurable functions $f\colon[0,1]\to\R$
such that  $f\in L^1(|x^*m_X|)$, for every $x^*\in X^*$. It is a B.f.s.\
for the \lq\lq same\rq\rq norm \eqref{norm-m} as used in $\lmx$ and contains $\lmx$
as a closed subspace, \cite[Ch.3, \S1]{okada-ricker-sanchez}.
The Copson operator $\ces^*$ was defined in the proof of Theorem \ref{L1}.
Whenever $X$ has a.c.\ norm and $\overline{\alpha}_X<1$ it is the dual
operator to $\ces\colon X\to X$.

The following  result is a summary of facts that occur in \cite{curbera-ricker3},
specialized to the \ce operator. Parts (a), (f) already occur in Proposition \ref{3.1}
and (b) also occurs in \cite[Theorem 1]{lesnik-maligranda-1}.
Part (k) is Theorem 3.1 of \cite{schep}; it provides an
alternate description of $\cx'$ to that given in \eqref{c*}.

\begin{pro}\label{all Indag}
Let $X\not=\linf$ be a r.i.\ space.
\begin{itemize}
\item[(a)] If $X$ has a.c.\  norm, then $\cx$ has a.c.\  norm and $\cx=\lmx$.
\item[(b)] If $X$ has the Fatou property, then $\cx$ has the Fatou property.
\item[(c)] If $X$ has the weak Fatou property, then $\cx$ has the weak Fatou property.
\item[(d)] If $X'$ is a norming subspace of $X^*$, then $\cx'$ is a norming subspace of $\cx^*$.
\item[(e)] If $X'$ is a norming subspace of $X^*$, then $\cx''=[\mathcal{C},X'']$.
\item[(f)] If $f\in \lmx$, then $f\in\cx$ and $\|f\|_{\lmx}= \|f\|_{\cx}$.
\item[(g)] If $f\in \cx$, then $f\in\wlmx$ and $\|f\|_{\wlmx}\le\|f\|_{\cx}$.
\item[(h)] If $f\in \wlmx$, then $f\in[\mathcal{C},X'']$ and $\|f\|_{[\mathcal{C},X'']}\le\|f\|_{\wlmx}$.
\item[(i)] $\cx''=\wlmx$ with equality of norms.
\item[(j)] If $X'$ is a norming subspace of $X^*$, then $\wlmx=[\mathcal{C},X'']$.
\item[(k)] If $X$ has a.c.\ norm, the Fatou property and satisfies $\overline{\alpha}_X<1$,
then $\cx'$ equals the ideal in $L^0$
generated by the range $\{\ces^*(f):f\in X'\}$ where $\ces^*$ acts in $X'$.
\end{itemize}

In the event that $X'$ is a norming subspace of $X^*$, there is equality of norms in (g) and (h).
\end{pro}

The following chain of inclusions, which refines \eqref{inclusions},
summarizes the situation (cf. (9) on p.199 of \cite{curbera-ricker3}):
\begin{equation}\label{eq fatou}
\lmx \subseteq \cx \subseteq \cx'' = \wlmx = \lmx'' \subseteq [\mathcal{C},X''] .
\end{equation}
If $X$ has a.c.\  norm, then the first and last containments are equalities
and the second containment an isometric embedding. On the other hand,  if $X$ has the Fatou property (i.e., $X=X''$),
then the second and last containments are equalities. Finally, in case $X$ has both a.c.\  norm and the Fatou
property (i.e., $X$ is weakly sequentially complete), then  all spaces involved coincide.

It should be stressed that the space $\wlmx$ plays a crucial role.
Recall that whenever a B.f.s.\ $X$ space does not have the Fatou property, then it is
always possible to identify its \textit{`Fatou completion'}, that is,
the smallest of all B.f.s.' which contain $X$ and have  the Fatou property,
 \cite[\S 71, Theorem 2]{zaanen1}. This space coincides with
$X''$. Proposition \ref{all Indag}(i)  shows that the space $\wlmx$ is
precisely the Fatou completion
of the \ce space $\cx$, whereas  $\lmx$ is the a.c.\  part of $\cx$.

We conclude with two relevant results. Recall (cf. Theorem \ref{reflexive}(a))
that  the space  $\cx$ is never reflexive. Weak sequential completeness of a B.f.s.\
is often a good replacement for the space failing to be reflexive.

\begin{pro}\label{5.2}
Let $X\not=\linf$ be a r.i\ space.
\begin{itemize}
\item[(a)]  If the integration operator $\imx\colon\lmx\to X$ is weakly compact,
then $\cx$ is weakly sequentially complete.

If, in addition, $\overline{\alpha}_X<1$, then  $\ces\colon X\to X$ is also weakly compact.

\item[(b)]  If  the integration operator
$\imx\colon\lmx\to X$ is completely continuous,
then $\cx$ is weakly sequentially complete.

If, in addition, $\overline{\alpha}_X<1$, then  $\ces\colon X\to X$ is also completely continuous.
\end{itemize}
\end{pro}

\begin{proof}
(a)  If  $\imx\colon\lmx\to X$ is weakly compact, then Corollary 2.3 of
\cite{curbera-ricker4} asserts that
$\lmx=\wlmx$ and hence, $\lmx$ has the Fatou property; see \eqref{eq fatou}. Being also a.c.,
it follows that  $\lmx$ is weakly sequentially complete.
Again according to \eqref{eq fatou} we then have  $\lmx=\cx =\wlmx$.

If, in addition, $\ces\colon X\to X$, then $\ces$ factorizes through $\lmx$ via $\imx$ and
so is itself also weakly compact.

(b) If  $\imx\colon\lmx\to X$ is completely continuous, then again it is known that
necessarily $\lmx=\wlmx$, \cite[Theorem 3.6]{delcampo-etal}. A similar argument as in the proof of (a) establishes
the result.
\end{proof}


\bibliographystyle{amsplain}


\end{document}